\documentclass[12pt]{amsart}

\usepackage{cite} 

\usepackage{amsmath,amssymb,url}
\usepackage{amsthm}
\usepackage{tikz,color,pgf}

\newtheorem{theorem}{Theorem}[section]
\newtheorem{corollary}[theorem]{Corollary}
\newtheorem{lemma}[theorem]{Lemma}
\newtheorem{proposition}[theorem]{Proposition}

\newcommand{\ignore}[1]{}

\title{Fibonacci-run graphs I: basic properties}

\date{\today}

\author{\"Omer E\u{g}ecio\u{g}lu}
\address{Department of Computer Science, University of California Santa Barbara 
\\Santa Barbara, California 93106, USA }
\email{omer@cs.ucsb.edu}

\author{Vesna Ir\v{s}i\v{c}}
\address{Faculty of Mathematics and Physics, University of Ljubljana
\\Slovenia \newline
Institute of Mathematics, Physics and Mechanics, Ljubljana, Slovenia}
\email{vesna.irsic@fmf.uni-lj.si}

\newcommand{\R}{\mathcal{R}}
\DeclareMathOperator{\adeg}{\overline{deg}}
\DeclareMathOperator{\diam}{diam}

\begin{document}
	

\begin{abstract}
Among the classical models for interconnection networks are hypercubes and Fibonacci cubes. 
Fibonacci cubes are induced subgraphs of hypercubes obtained by 
restricting the vertex set to 
those binary strings which do not contain consecutive 1s, 
counted by Fibonacci numbers. 
Another set of binary strings which are counted by Fibonacci numbers are those with 
a restriction on the runlengths.
Induced subgraphs of the hypercube on the
latter strings as vertices define {\em Fibonacci-run 
	graphs}. They have the same number of vertices as Fibonacci cubes, but fewer edges and
different connectivity properties.

We obtain properties of Fibonacci-run graphs including the number of edges, the analogue of the fundamental recursion,
the average degree of a vertex, 
Hamiltonicity, and special degree sequences, the number of hypercubes they contain. A detailed study of the 
degree sequences of Fibonacci-run graphs is interesting in its own right and is reported 
in  a companion paper.


\medskip\noindent
\textbf{AMS Math.\ Subj.\ Class.\ (2020)}: 05C75, 05C30, 05C12, 05C40, 05A15

\end{abstract}
\maketitle
\section{Introduction}
\label{sec:intro}

The \emph{$n$-dimensional hypercube} $Q_n$ is the graph on the vertex set $$\{0, 1\}^n = \{ v_1 v_2 \ldots v_n ~|~ \; v_i \in \{0,1\} \},$$ where 
two vertices $v_1 v_2 \ldots v_n$ and $u_1 u_2 \ldots u_n$ are adjacent if $v_i \neq u_i$ for exactly 
one index $i \in [n]$. 
In other words, vertices of $Q_n$ are all possible strings of length $n$ consisting only of $0$s and $1$s, 
and two vertices are adjacent if and only if they differ in exactly one coordinate or ``bit''. 
Clearly, $|V(Q_n)| = 2^n$, and $|E(Q_n)| = n 2^{n-1}$.

A well studied subfamily of hypercubes are \emph{Fibonacci cubes} $\Gamma_n$ 
which were introduced by Hsu~\cite{hsu1993} as an alternate interconnection topology. 
The vertex set $\Gamma_n$ consists of the
\emph{Fibonacci strings} of length $n$, 
$$\mathcal{F}_n = \{ v_1 v_2 \ldots v_n \in \{0,1\}^n ~|~ \; v_i v_{i+1} = 0 ~, 
i \in [n-1]\}~,$$
and two vertices are adjacent if and only if they differ in exactly one coordinate. 
In other words, 
$\Gamma_n$ is the subgraph of $ Q_n$, 
induced by the vertices that do not contain consecutive $1$s. 
This family of graphs has proved to have an interesting structure and was 
studied extensively. 

As mentioned in the seminal paper, the Fibonacci cubes can be viewed as 
an interconnection topology and can be applied to fault-tolerant computing~\cite{hsu1993}. 
A survey of a variety of properties of Fibonacci cube graphs~\cite{klavzar2013-survey} 
presents not just the
fundamental decomposition of Fibonacci cubes, 
but also its Hamiltonian properties, and many enumeration results. 
For some very recent results, see~\cite{savitha+2020, saygi+2019, azarija+2018, mollard2017}. Following similar ideas, other graph families have also been introduced and studied, such as Lucas cubes~\cite{munarini+2001}, generalized Fibonacci cubes~\cite{ilic+2012}, 
$k$-Fibonacci cubes~\cite{Egecioglu2020} and daisy cubes~\cite{klavzar+2019}.

Note that a Fibonacci cube can equivalently be defined 
by adding $00$ to the end of the binary representation 
of every vertex. 
We can call such binary strings {\em extended Fibonacci strings}.
Actually, in an extended Fibonacci string,
the rightmost zero corresponds to the Fibonacci 
number  $f_0$ and the second to last zero corresponds to $ f_1$ in the 
encoding of the so called Zeckendorf or canonical representation 
of integers~\cite{zeckendorf1972}. In this representation, $f_0$ is not needed since it is zero, and $f_1$ is not needed because $ f_2$ is 
already 1, and the inclusion of another 1 would
prevent uniqueness of the Zeckendorf expansion. We can 
call this representation with the additional 0s the {\em extended Zeckendorf representation.}

With this interpretation
$$
V(\Gamma_n) = \{ w 00 ~|~ w \in \mathcal{F}_n \}
$$ 
and two vertices are adjacent 
if they differ in exactly one coordinate. 
Recall that the \emph{Hamming distance} $H(u,v)$ between strings $u, v \in \{0,1\}^n$ is the number of coordinates 
in which $u$ and $v$ differ. Using this, the edge set of $\Gamma_n$ can be described as 
$$
E(\Gamma_n) = \{ \{ u00, v00 \} ~|~ H ( u,v ) = 1 \} ~.
$$

Observe that extended  Fibonacci strings defined above, 
together with the null word $ \lambda$ and the singleton $0$, 
are generated freely (as a monoid) by
the infinite alphabet
$$
F = 0, 100, 10100, 1010100, \ldots
$$
This means that every $v \in V(\Gamma_n)$ can be written 
uniquely as a concatenation of zero or more strings from $F$.

With this in mind we introduce a new family of graphs. Instead of 
considering extended Fibonacci strings as the vertex set to define our communication network, 
we consider \emph{run-constrained binary strings}. 
These are strings of $0$s and $1$s, in which every run (sometimes called a {\em  block}) 
of $1$s appearing in the word is immediately followed by a strictly longer run of $0$s. 
Such run-constrained strings, together with the null word $\lambda$ and the singleton $0$, 
are generated freely by the letters from the alphabet
\begin{equation}\label{R}
R = 0,100, 11000,1110000, \ldots 
\end{equation}
Note that run-constrained strings of length $ n \geq 2$ must end with $00$. 

These strings  allow us to define the \emph{Fibonacci-run graph} $\R_n$, parametrized by 
$ n \geq 0$, in the following fashion. The vertex set of $\R_n$ is 
$$
V(\R_n) = 
\{ w 00 ~|~ w00 \mbox{ is a run-constrained string of length $n+2$} \}~,
$$ 
and its edge set is 
$$
E(\R_n) = \{ \{ u00, v00 \} ~|~ H ( u,v ) = 1 \}~.
$$
The term Fibonacci-run graph makes sense, 
since $|V(\R_n)| = |V(\Gamma_n)|$ is 
the $(n+2)$-th Fibonacci number $f_{n+2}$, as we will later show in Section~\ref{sec:basic}.

Clearly, $\R_n$ is a subgraph of $Q_{n+2}$, but it 
is more natural to see it as a subgraph of $Q_n$ (after suppressing the tailing $00$ in the 
vertices of $\R_n$). In fact, from now on, we will view the vertices of a Fibonacci-run 
graph without the trailing pair of zeros as 

$$ 
V(\R_n) = 
\{ w ~|~ w00 \mbox{ is a run-constrained binary string of length $n+2$}\} ~.
$$
This is the same kind of a convention as viewing $\Gamma_n$ as a subgraph of $ Q_{n+2}$ if 
we think of the vertices as extended Fibonacci strings, or as a subgraph of $Q_n$ as usual by 
suppressing the trailing 00 of the vertex labels in the extended Zeckendorf representation.

Additionally, we define 
$\R_0$ to be isomorphic to $K_1$, and its only vertex corresponding to the label $00$, 
which after the removal of the trailing pair of zeros
corresponds to the null word.

We start the study of the Fibonacci-run graphs $ \R_n$ by considering the basic properties 
of these graphs, such as the number of edges, diameter, decomposition, Hamiltonicity, 
asymptotic average degree and compare these graph parameters with the known results for $ \Gamma_n$.

The rest of the paper is organized as follows.
After the preliminaries in Section~\ref{sec:prelim}, we consider the nature of the 
run-constrained binary strings in Section~\ref{sec:basic}, and calculate the number of 
edges of $\R_n$. This is followed by the graph-theoretic decomposition of $\R_n$, which is actually 
what we make use of for our 
results on the number of edges. In Section~\ref{sec:diam}, we consider the diameter of Fibonacci-run graphs,
followed by asymptotic results in Section~\ref{sec:asymptotic}. In Section~\ref{sec:degree},
we use the properties of run-constrained strings to calculate the generating functions of
the distribution of vertices with special types of degrees in $\R_n$. 
In Section~\ref{sec:hypercubes}, we prove that
Fibonacci-run graphs are partial cubes for only small values of $n$, and provide a count of the 
induced hypercubes $Q_k$ contained in $\R_n$. Finally, in Section
\ref{sec:hamilton} we study the Hamiltonicity of $\R_n$, and 
conclude with conjectures and further directions for research. 
Figures of a few large Fibonacci-run graphs are given in the 
Appendix.

\section{Preliminaries}
\label{sec:prelim}

In this section, we give definitions, provide notation, and indicate the 
known results that are needed in the paper. 
To avoid possible confusion, recall that \emph{Fibonacci numbers} are defined as 
$f_0 = 0$, $f_1 = 1$, and $f_n = f_{n-1} + f_{n-2}$ for $n \geq 2$. 
They are closely related to the \emph{golden ratio}
$\varphi = (1 + \sqrt{5})/2$, as  $\lim_{n \to \infty} f_{n+1}/f_n = \varphi$. 
In general, it can be shown that 
$\lim_{n \to \infty} f_{n+k}/f_n = \varphi^k$ for any integer $k$.

The \emph{Hamming weight} of a binary string $u$ is 
the number of $1$s in $u$, denoted by $|u|_1$. If $a, b$ are strings, 
then $ab$ denotes the concatenation of those two strings. 
Similarly, for a set of strings $B$, we set $aB = \{ab ~|~ \; b \in B\}$.

The degree of a vertex $v$ is the number of neighbors of $v$,
denoted by $\deg(v)$. Since each neighbor of $v \in V(\R_n)$ is obtained either by changing a $0$ in $v$ into a $1$ 
or by changing a $1$ in $v$ into a $0$, we can distinguish between the 
\emph{up-degree} $\deg_{up} (v)$, and \emph{down-degree} $\deg_{down} (v)$ of the vertex $v$. 
The first is the number of neighbors of $v$ which are obtained by switching a 
$0$ in $v$ into a $1$, while the second is the number of neighbors obtained by 
switching a $1$ in $v$ into a $0$. 
Clearly, $\deg(v) = \deg_{up} (v) + \deg_{down}(v)$. 
Note that in $ \R_n$, $\deg_{down}(v)$ is not necessarily 
equal to the Hamming weight of $v$ because of 
the constraints on the run-lengths that must hold. 

We keep track of the degree sequences of our graphs $\R_n$ as the 
coefficients of a polynomial. This polynomial
is called 
the \emph{degree enumerator polynomial} of the graph. The coefficient of $x^i$ in 
the degree enumerator polynomial
is
the number of vertices of degree $i$ in $\R_n$. Similar polynomials are defined to keep track of the 
up- and down-degree sequences as well (see the beginning of Section~\ref{sec:degree} for an example).

The average degree of a graph $G$ is
$$
\adeg(G) = \frac{1}{|V(G)|} \sum_{v \in V(G)} \deg(v) 
= \frac{2 |E(G)|}{|V(G)|}~.
$$ 

The \emph{diameter} of $G$ is the maximum
distance between pairs of vertices of the graph, that 
is $$\diam(G) = \max \{ d(u,v) ~|~ \; u, v \in V(G)\}~.$$
We use the notation $ d_G(u,v)$ when we need to emphasize that the distance is relative to $G$.

A \emph{partial cube} is an isometric subgraph of a hypercube. Recall that a subgraph $H$ 
of a graph $G$ is an \emph{isometric subgraph} if for every $u, v \in V(H)$ we have 
$d_H(u,v) = d_G(u,v)$. A graph $G$ is a \emph{median graph} if every triple $u, v, w$ of its vertices 
has a unique median $x$, that is a vertex $x$ with the properties 
$d(u,x) + d(x,v) = d(u,v)$, $d(v,x) + d(x,w) = d(v,w)$, and $d(u,x) + d(x,w) = d(u,w)$. 
Recall that hypercubes are median graphs, and that median graphs are partial cubes.

\section{The number of edges}
\label{sec:basic}

\begin{lemma}
	\label{lem:vertices-bij}
	If $n \geq 1$, then $|V(\R_n)| =  f_{n+2}$.
\end{lemma}

\begin{proof}
Since we know that 
	$|V(\Gamma_n)| = f_{n+2}$~\cite{hsu1993, klavzar2013-survey}, it suffices to provide a bijection between vertex sets of $\Gamma_n$ and $\R_n$. 
For the purpose of this proof, we view both vertex sets with additional $00$ at the end of each string. 
	
	An explicit bijection between Fibonacci strings and run-constrained strings
	is obtained by setting
	\begin{eqnarray*}
		\Phi(0) &=&0 \\
		\Phi(100) &=&100 \\
		\Phi(10100) &=&11000 \\
		\Phi(1010100) &=&1110000  , 
	\end{eqnarray*}
	etc., a bijection between the alphabets $F$ and $R$, and extending $\Phi$ to full words via the unique factorization.
	So for example, starting with a Fibonacci string $w = 10010100001010100100 $, we find $\Phi(w)$ as 
	\begin{eqnarray*}
		\Phi \big(w \big)
		& = & \Phi \Big((100)(10100)(0)(0)(1010100)(100)\Big)\\
		&=& (100)(11000)(0)(0)(1110000)(100) \\
		& = & 10011000001110000100.
	\end{eqnarray*}
	This clearly yields a bijection, thus $|V(\R_n)| = |V(\Gamma_n)| = f_{n+2}$.
\end{proof}

\begin{corollary}
	\label{cor:vertices-weight}
	The number of vertices in $\R_n$ of Hamming weight $w$, $0 \leq w \leq \lceil n/2 \rceil$ is 
	$$\binom{n-w+1}{w}~.$$
\end{corollary}

\begin{proof}
	The result follows from the bijection $\Phi$ 
between vertex sets of $\R_n$ and $\Gamma_n$ (as it preserves the number of $1$s in a string), 
and the classical result that the 
number of Fibonacci strings of 
length $n$ and  Hamming weight $w$ is $\binom{n-w+1}{w}$~(see, for example \cite{saygi2019}).
\end{proof}

Figure~\ref{fig:fibonacci_run} depicts the first four Fibonacci graphs $\Gamma_n$ and the 
first four Fibonacci-run graphs $\R_n$. As noted, we omit the ending $00$ of the vertex labels in both. Note that the graphs $\Gamma_n$ and $\R_n$ are isomorphic for $n \in [4]$. 
However, the bijection $\Phi$ between the vertex sets of these graphs is not in general a graph isomorphism. 
For example, for the adjacent vertices $u= 101$ and $v=001$ in $\Gamma_3$, we have
$\Phi(u) = 110$, $ \Phi(v) = 001$ with 
$H(\Phi(u), \Phi(v))=3 $  in $\R_3$, so $\Phi(u)$ and $\Phi(v)$ 
are not adjacent in $\R_3$. However as it turns out the two graphs are isomorphic 
anyway (under a different graph isomorphism), see Figure~\ref{fig:fibonacci_run}.

\begin{figure}[ht]
	\begin{center}
		\includegraphics[width=0.7\textwidth]{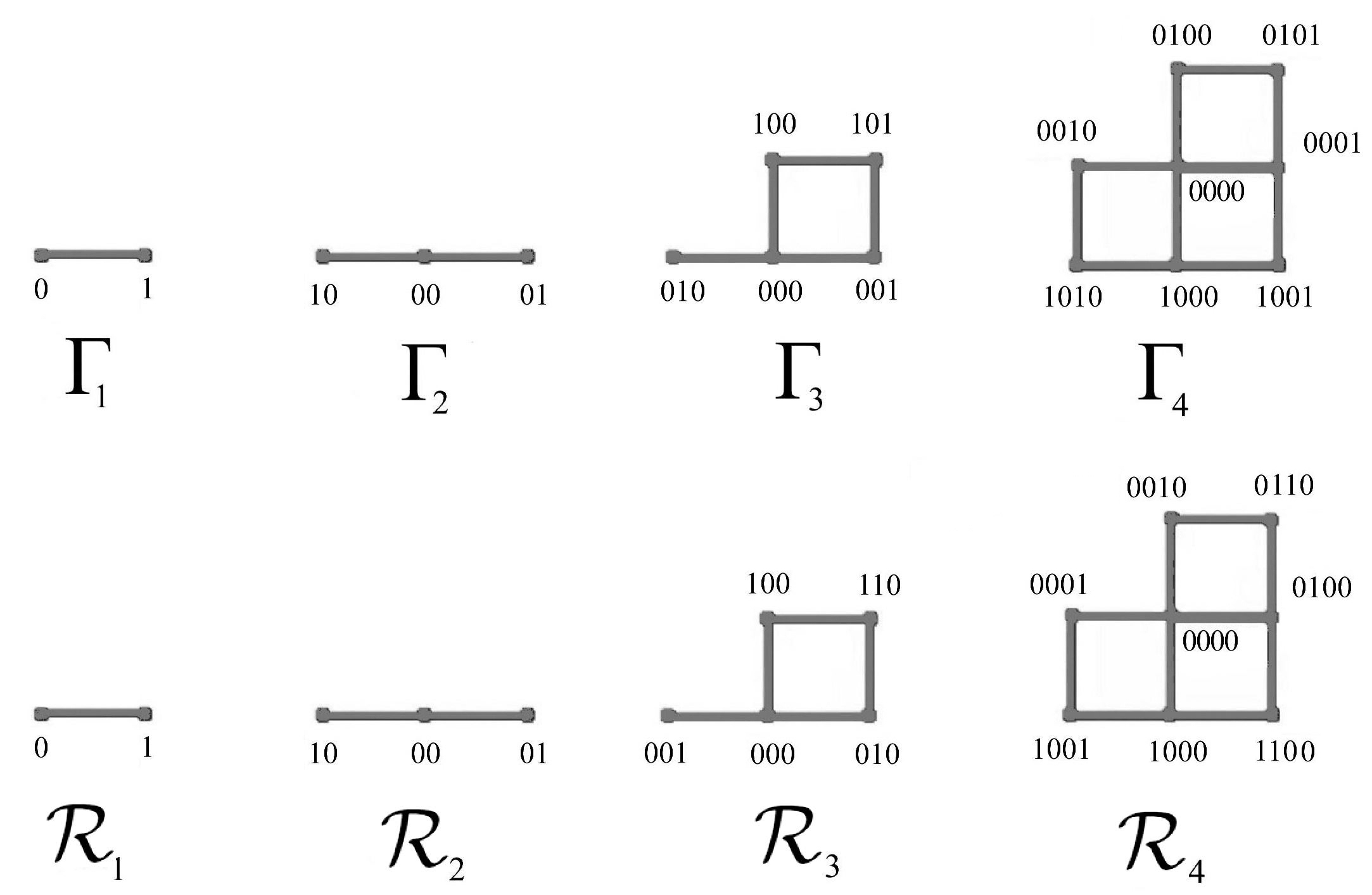}
	\end{center}
	\caption{Graphs $\Gamma_n$ and $\R_n$ for $n \in [4]$.}
	\label{fig:fibonacci_run}
\end{figure}


In the following, we prove that the graphs $\Gamma_n$ and $\R_n$ are never isomorphic for $n \geq 5$. 
Recall~\cite{klavzar2013-survey} that the number of edges of $\Gamma_n$ is given by
\begin{equation}\label{eq:esges-fib}
|E(\Gamma_n)| = \frac{1}{5} \left(2(n+1) f_n + n f_{n+1}\right),
\end{equation}
with generating function
\begin{equation}\label{eq:ogf-fib}
\sum_{n \geq 1} |E(\Gamma_n)| t^ n = \frac{t}{(1-t-t^2)^2}  = t + 2 t^2 + 5 t^3 + 10 t^4 + 20 t^5 + 38 t^6 + \cdots \end{equation}

\begin{lemma}
	\label{lem:edges}
	If $n \geq 5$, then 
\begin{equation}\label{inedges}
|E(\R_n)| = |E(\Gamma_n)| - |E(\Gamma_{n-4})|,
\end{equation}
	with generating function $$\sum_{n \geq 1} |E(\R_n)| t^ n = 
	\frac{t(1-t^4)}{(1-t-t^2)^2} = t + 2 t^2 + 5 t^3 + 10 t^4 + 19 t^5 + 36 t^6 + \cdots$$
\end{lemma}

\begin{proof}
	For the proof of the expression~\eqref{inedges}, see Section~\ref{sec:decomposition},  Corollary~\ref{cor:edges-recursion}.
	
	The ordinary generating function for $|E(\R_n)|$ 
follows directly from the recursion, the properties of graphs $\R_1, \R_2, \R_3, R_4$, and~\eqref{eq:ogf-fib}:
	\begin{align*}
		\sum_{n \geq 1} |E(\R_n) | t^ n  = & ~ t + 2 t^2 + 5 t^3 + 10 t^4 + 
\sum_{n \geq 5} (|E(\Gamma_n)| - |E(\Gamma_{n-4})|) t^ n  \\
		 = & ~ t + 2 t^2 + 5 t^3 + 10 t^4 + \sum_{n \geq 5} |E(\Gamma_n) |t^ n - 
t^4 \sum_{n \geq 1} | E(\Gamma_n)|  t^ n  \\
		 = & ~ t + 2 t^2 + 5 t^3 + 10 t^4 
		   + \frac{t}{(1-t-t^2)^2} \\
& ~ - (t + 2 t^2 + 5 t^3 + 10 t^4) 
		  - t^4 \frac{t}{(1-t-t^2)^2} \\
		 = & ~ \frac{t-t^5}{(1-t-t^2)^2} ~.
	\end{align*}
\end{proof}

In particular, an analytic expression for the number of edges of $ \R_n$ in terms of the 
Fibonacci numbers is given in~\eqref{edge_formula} in Corollary~\ref{cor:edges-recursion}.

\begin{corollary}
	\label{cor:non-iso}
	If $n \geq 5$, then the graphs $\R_n$ and $\Gamma_n$ are not isomorphic.
\end{corollary}

\begin{proof}
	If $n \geq 5$, then $|E(\Gamma_{n-4})| \geq 1$, and therefore from~\eqref{inedges} 
in Lemma~\ref{lem:edges} it follows that $ | E(\R_n)|  <|  E(\Gamma_n)| $, thus the 
graphs cannot be isomorphic.
\end{proof}

Clearly, the Fibonacci-run graphs are bipartite. The parts of the bipartition are simply obtained 
by separating the vertices of odd and even Hamming weight. Thus we have the following easy inequality 
for the vertex independence number $\alpha(\R_n)$.

\begin{lemma}
	\label{lem:independence}
	If $n \geq 1$, then $\alpha(\R_n) \geq \left \lceil \frac{f_{n+2}}{2} \right \rceil$.
\end{lemma}

\section{Decomposition of $\R_n$}
\label{sec:decomposition}

It is well known that Fibonacci cubes have a simple, and very useful, decomposition~\cite{hsu1993}. 
Namely, $\Gamma_n$ can be partitioned into subgraphs $0 \Gamma_{n-1}$ and $10\Gamma_{n-2}$, 
with a perfect matching between $ \Gamma_{n-2}$ and its copy $ 0 \Gamma_{n-2}$ in $ \Gamma_{n-1}$.
This immediately implies a formula for $|V(\Gamma_n)|$, $|E(\Gamma_n)|$, 
existence of a Hamiltonian path in $\Gamma_n$, and other 
interesting properties~\cite{hsu1993, klavzar2013-survey}. 
With this in mind, we aim to find a similar decomposition of the Fibonacci-run graphs. 

\begin{lemma}
	\label{lem:part-vertices}
	The vertex set of a Fibonacci-run graph $\R_n$ can be partitioned into $$ \bigcup_{k = 0}^{\lceil n/2 \rceil - 1} 1^k 0^{k+1} V(\R_{n - (2k+1)}) \cup 1^{\lceil n/2 \rceil} 0^{\lfloor n/2 \rfloor} V(\R_0)~.$$
\end{lemma}

\begin{proof}
	The vertices of $\R_n$ are run-constrained strings of length $n+2$, all ending with $00$. Each of them starts either with a $0$ or with a run of $1$s (followed by a longer run of $0$s). The maximal number of $1$s in such a string is $\lceil n/2 \rceil$. Thus every string in $V(\R_n)$ can be uniquely written as $1^k 0^{k+1} w$, where $0 \leq k \leq \lceil n/2 \rceil$, and $w$ is a run-constrained string of length $n+1-2k$. If $k < \lceil n/2 \rceil$, then $w \in V(R_{n-(2k+1)})$, while the last sting can be viewed as $1^{\lceil n/2 \rceil} 0^{\lfloor n/2 \rfloor} V(\R_0)$. This yields the 
described decomposition.
\end{proof}

It follows from Lemma~\ref{lem:part-vertices} that $\R_n$ contains the 
following graphs as subgraphs 
$$ 
0 \R_{n-1}, 100 \R_{n-3}, 11000 \R_{n-5}, 1110000 \R_{n-7} , \ldots 
$$ 
The last graph in this sequence is $1^{m-1} 0^{m} \R_1$ if $n = 2m$, 
and $1^{m-1} 0^{m} \R_0$ if $n = 2m-1$. We augment this list by one more subgraph 
of $\R_n$ consisting of a single vertex: $1^m 0^{m} \R_0$ if $n = 2m$, and $1^{m} 0^{m-1} \R_0$ 
if $n = 2m-1$. 

If we denote this partition into subgraphs with $+$, then we obtain the following for $n \in [8]$:
\begin{align*}
	\R_1 &= 0 \R_0 + 1 \R_0\\
	\R_2 &= 0 \R_1 + 10\R_0\\
	\R_3 &= 0\R_2 +  100\R_0  + 110\R_0\\
	\R_4 &= 0\R_3 +  100\R_1 + 1100 \R_0\\
	\R_5 &= 0\R_4 +  100\R_2 + 11000\R_0 + 11100 \R_0\\
	\R_6 &= 0\R_5 +  100\R_3 + 11000\R_1 +111000 \R_0 \\
	\R_7 &= 0\R_6 + 100\R_4 +11000\R_2 + 1110000\R_0 + 1111000 \R_0\\
	\R_8 &= 0\R_7 + 100\R_5 +11000\R_3 + 1110000\R_1 +11110000\R_0
\end{align*}

A schematic representation of the decomposition of graphs $\R_1, \ldots, \R_5$ is shown 
in Figure~\ref{fig:decomposition-2}.

\begin{figure}[!ht]
	\begin{center}
		\includegraphics[width=0.8\textwidth]{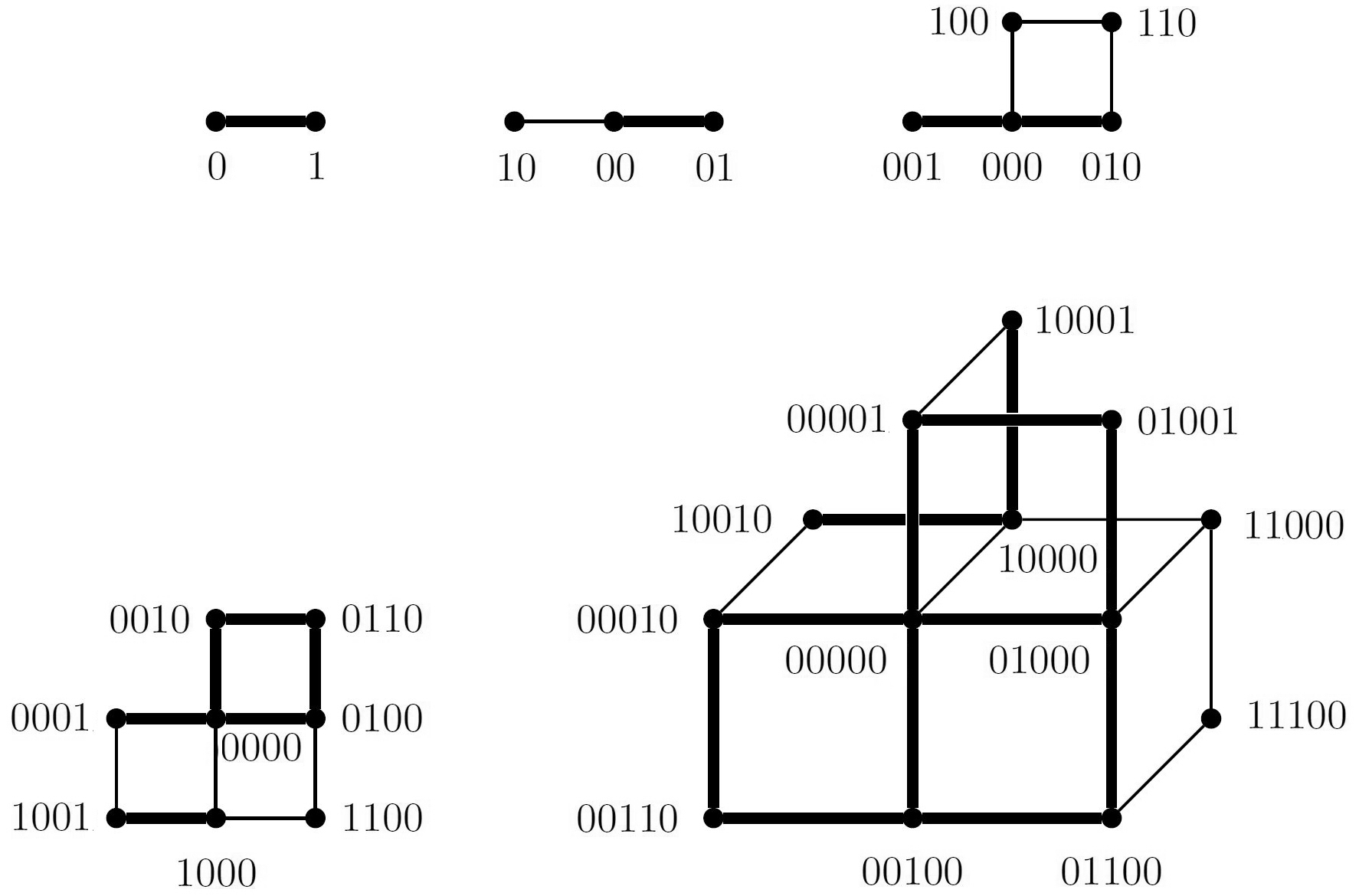}
	\end{center}
		\caption{The graphs $\R_n$, for $n \in [5]$, with their decomposition schematically shown with thicker edges.}
		\label{fig:decomposition-2}
\end{figure}

Next, we consider the edges between the mentioned subgraphs of $\R_n$. Inside 
each subgraph $1^k 0^{k+1} \R_{n - (2k+1)}$ we inherit the edges from the 
graph $\R_{n - (2k+1)}$. Between subgraphs $0 \R_{n-1}$ 
and $1^k 0^{k+1} \R_{n - (2k+1)}$ for $1 \leq k \leq \lceil n/2 \rceil -1$, we 
have an edge if and only if the two strings differ in exactly one coordinate. 
This means that the edges are of the form $0 1^{k-1} 0^{k+1} w \sim 1^k 0^{k+1} w$, where 
$w \in V(\R_{n - (2k+1)})$. 
Therefore each pair of subgraphs $0 \R_{n-1}$ and $1^k 0^{k+1} \R_{n - (2k+1)}$ 
for $1 \leq k \leq \lceil n/2 \rceil -1$ yields exactly $|V(\R_{n - (2k+1)})|$ edges 
in $\R_n$. On the other hand, between subgraphs $1^k 0^{k+1} \R_{n - (2k+1)}$ and 
$1^{\ell} 0^{\ell+1} \R_{n - (2\ell+1)}$ for $k < \ell$ 
and $1 \leq k,\ell \leq \lceil n/2 \rceil -1$, we can have an edge only if 
$\ell = k + 1$. To be precise, the edges between those two subgraphs are of the form 
$1^{k+1} 0^{k+2} w \sim 1^k 0^{k+3} w$, where $w \in V(\R_{n-(2k+3)})$, so we 
get exactly $|V(\R_{n-(2k+3)})|$ edges between appropriate subgraphs. The only remaining edges to 
study are between the vertex $1^{\lceil n/2 \rceil} 0^{\lfloor n/2 \rfloor} V(\R_0)$ 
and the other parts of the graph. But this vertex has exactly two neighbors 
in $\R_n$: $1^{\lceil n/2 \rceil -1} 0^{\lfloor n/2 \rfloor +1} V(\R_0)$ 
and $0 1^{\lceil n/2 \rceil -1} 0^{\lfloor n/2 \rfloor} V(\R_0)$. 

As examples,
see Figure~\ref{fig:decomposition} for a schematic representation of the decomposition 
of $\R_6$ and $\R_7$.

\begin{figure}[ht]
	\begin{center}
		\includegraphics[width=0.95\textwidth]{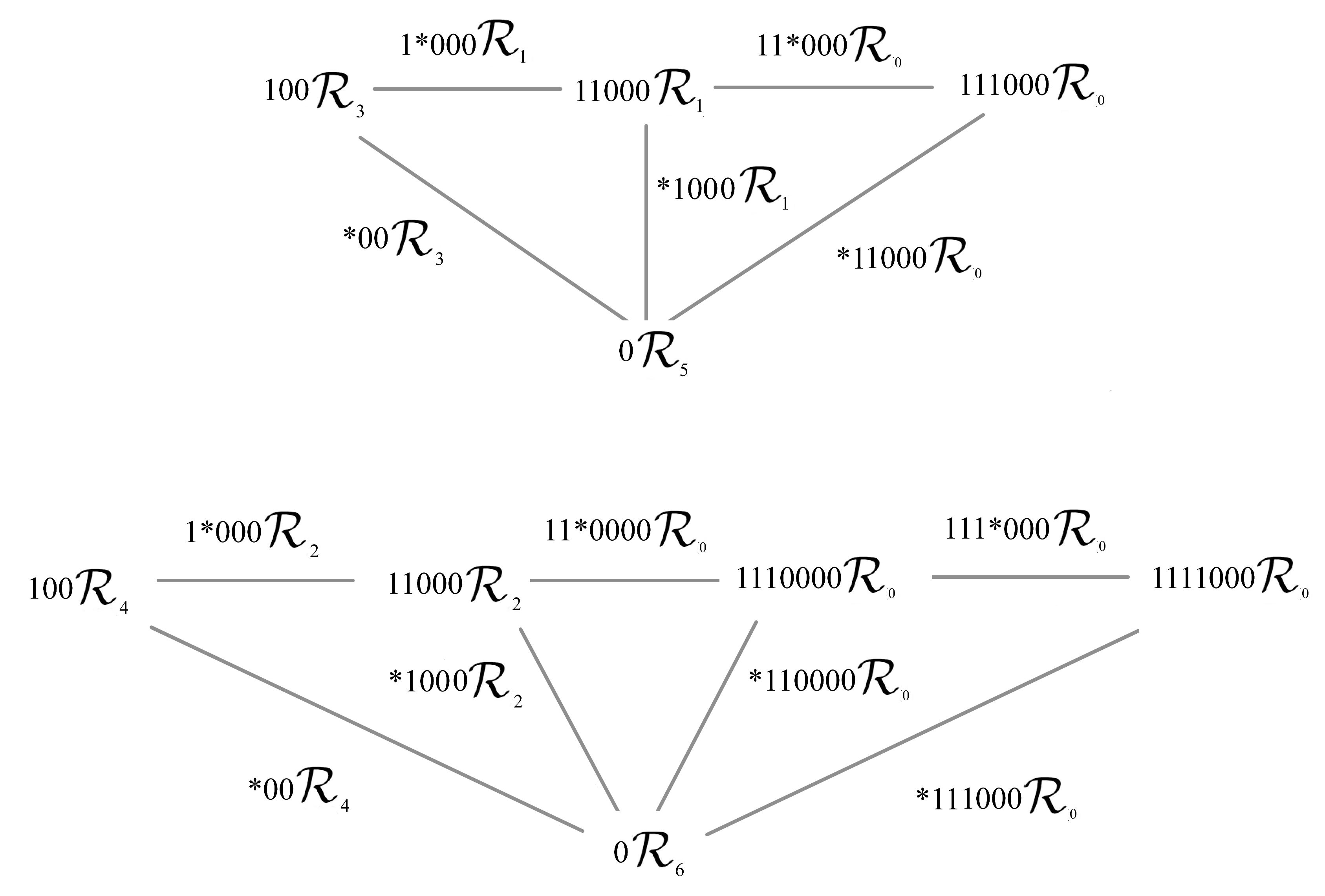}
	\end{center}
	\caption{The decomposition of $ \R_6$ (above) and $\R_7$ (below). The edges are symbolically marked with the lines indicating the edges present between the parts.}
	\label{fig:decomposition}
\end{figure}

We already know (see Figure~\ref{fig:fibonacci_run}), that $|E(\R_1)| = 1$, $|E(\R_2)| = 2$, 
$|E(\R_3)| = 5$, and $|E(\R_4)| = 10$. 
For $n \geq 5$, we can use the above argument to obtain the following 
\begin{equation}
	\label{eq:edges-decomposition}
	|E(\R_n)| = \sum_{k=0}^{\lceil n/2 \rceil -1} |E(\R_{n - (2k+1)})| 
+ |V(\R_{n-3})| + 2 \sum_{k = 2}^{\lceil n/2 \rceil -1} |V(\R_{n - (2k+1)})| + 2.
\end{equation}

\begin{lemma}
	\label{lem:edges-decomposition}
	If $n \geq 4$, then 
	$$|E(\R_n)| = |E(\R_{n-1})| + |E(\R_{n-2})| + f_{n-1} + f_{n-3}~.$$
\end{lemma}

\begin{proof}
	We know from Lemma~\ref{lem:vertices-bij} that $|V(\R_n)| = f_{n+2}$. 
Set $e_n = |E(\R_n)|$. If $n \geq 7$, we can use the recursion 
from~\eqref{eq:edges-decomposition} for $n$ and $n-2$:
\begin{eqnarray*}
e_n & = & e_{n-1} + \sum_{k=1}^{\lceil n/2 \rceil -1} 
e_{n - (2k+1)} + |V(\R_{n-3})| + 2 |V(\R_{n-5})| \\
& & \hspace*{2.5cm} + ~2  \sum_{k = 3}^{\lceil n/2 \rceil -1} |V(\R_{n-(2k+1)})| + 2\\
	& = &e_{n-1} + f_{n-1} + f_{n-3} + e_{n-2}~.
\end{eqnarray*}
	For $n \in \{4,5,6\}$, we can determine the number of edges $e_n$ 
from~\eqref{eq:edges-decomposition}, and check by hand that the recursion 
$e_n = e_{n-1} + e_{n-2} + f_{n-1} + f_{n-3}$ gives the same result.	
\end{proof}

\begin{corollary}
	\label{cor:edges-recursion}
	If $n \geq 5$, then $$|E(\R_n)| = |E(\Gamma_n)| - |E(\Gamma_{n-4})|~.$$
\end{corollary}

\begin{proof}
	We can easily check that the recursion holds for $n = 5$. Now 
let $n \geq 6$, and let again $e_n = |E(\R_n)|$. By Lemma~\ref{lem:edges-decomposition}, the 
induction hypothesis, the equation~\eqref{eq:esges-fib}, 
and the relation $f_n = f_{n-1} + f_{n-2}$, it follows that 
	
\begin{eqnarray}\nonumber
e_n &= &(e_{n-1} - e_{n-5}) + (e_{n-2} - e_{n-6}) + f_{n-1} + f_{n-3} \\ \label{edge_formula}
&  = & (3 n + 4 ) f_{n-6} + (5 n + 6 ) f_ {n-5}~. 
\end{eqnarray}
	On the other hand, we can also calculate that 
	$$|E(\Gamma_n)| - |E(\Gamma_{n-4})| = (3 n + 4) f_{n-6} + (5 n  + 6) f_ {n-5} ~,$$
	which completes the proof.
\end{proof}

To conclude the section, we observe that from the decomposition it 
follows that each vertex in $\R_n \setminus 0 \R_{n-1}$ has exactly one 
unique neighbor in $0 \R_{n-1}$. More precisely, let $\varphi \colon \R_n \setminus 0 \R_{n-1} \to 0 \R_{n-1}$ be defined as 
\begin{align*}
1^k 0^{k+1} w & \mapsto 0 1^{k-1} 0^{k+1} w, \; w \in \R_{n-2k-1}, \; 1 \leq k \leq \lceil n/2 \rceil - 1,\\
1^{\lceil n/2 \rceil} 0^{\lfloor n/2 \rfloor} & \mapsto 0 1^{\lceil n/2 \rceil -1} 0^{\lfloor n/2 \rfloor}.
\end{align*}
Clearly, the function $\varphi$ is injective. Moreover, for every 
$u, v \in \R_n \setminus 0$ it satisfies $ u \sim v \iff \varphi(u) \sim \varphi(v)$.


\section{Diameter}
\label{sec:diam}

The diameter of Fibonacci cubes is well known~\cite{hsu1993}, and equals 
$\diam(\Gamma_n)=n$. 
Determining the diameter of Fibonacci-run graphs turns out to be a rather difficult task. 
Clearly, $\diam(\R_n) = n$ for $n \in [4]$. 
The exact values of the diameter computed by brute force are 
shown in Figure~\ref{fig:diameter_plot} for $ n \leq 30$.
For general $n$ we present a lower bound on $\diam(\R_n)$, 
and a conjecture on the actual 
value of the diameter.

We first prove the following lemma, which will be generalized later.

\begin{lemma}
	\label{lem:diam}
	If $ n = \frac{1}{2} (r^2 +3r -2)$, then $\diam(\R_n) \geq 
	n- \left\lfloor \frac{r-1}{2} \right\rfloor$.
\end{lemma}
\begin{proof}
	It suffices to find two vertices $u,v \in V(\R_n)$ with 
	$d(u,v) = n- \lfloor \frac{r-1}{2} \rfloor$. 

	If $r$ is even, take 
	\begin{align*}
	u & = 10^21^30^4 \ldots 1^{r-1}0^r1^{r/2}0^{r/2+1},\\
	v & = 01^20^31^4 \ldots 0^{r-1} 1^r 0^{r+1},
	\end{align*}
	where the trailing $00$ is {\em not} omitted. In this case, 
	\begin{align*}
		n &= 1+2 + \cdots + r + (r-1) = \frac{1}{2} (r^2 + 3r -2),\\
		d(u,v) &= 1+ 2+  \cdots + r + \frac{r}{2} = 
n - \left(\frac{r}{2} - 1\right) = n- \left\lfloor \frac{r-1}{2} \right\rfloor.
	\end{align*}
	If $r$ is odd, then we take
	\begin{align*}
	u & = 10^21^30^4 \ldots 1^{r}0^{r+1},\\
	v & = 01^20^31^4 \ldots 0^r 1^{(r+1)/2 - 1} 0^{(r+1)/2 + 1},
	\end{align*}
	where the last $00$ is again not omitted. In this case, we get
	\begin{align*}
		n &= 1+2 + \cdots + r + (r-1) = \frac{1}{2} (r^2 + 3r -2), \\
		d(u,v) &= 1+ 2+  \cdots + r + \frac{r-1}{2} = n - \left( \frac{r-1}{2}\right) 
= n- \left\lfloor \frac{r-1}{2} \right\rfloor. 
	\end{align*}
\end{proof} 

\begin{corollary}
	\label{cor:diam_special}
	If $ n = \frac{1}{2} (r^2 +3r -2)$, then $\diam(\R_n) \geq n - \sqrt{n/2}$.
\end{corollary}

\begin{proof}
	This can be verified simply by checking that 
	$\sqrt{ \frac{n}{2}} \geq \frac{r-1}{2} \geq \left\lfloor \frac{r-1}{2} \right\rfloor	$.
\end{proof}

\begin{theorem}
	\label{thm:diam_lower}
	If $n \geq 1$, then $\diam(\R_n) > n-\sqrt{2n}$.
\end{theorem}

\begin{proof}
	If $n = \frac{1}{2} (r^2 +3r -2)$ for some integer $r$, then the result holds. Otherwise, 
we have 
	$$\frac{1}{2} (r^2 +3r -2) < n < \frac{1}{2} ((r+1)^2 +3(r+1) -2)~,$$
	which implies
	$$n = \frac{1}{2} (r^2 +3r -2) + D, \; \mbox{  with   } 1 \leq D \leq r+1~.$$
	We again aim to construct vertices $u, v \in V(\R_n)$ which differ on 
at least $n-\sqrt{2n}$ coordinates. The idea is to extend the last run of $r+1$ 
zeros in $u$ or $v$ from the proof of Lemma~\ref{lem:diam}, and add the 
maximum allowed number of $1$s followed by appropriate number of $0$s in the other vertex.
	
	If $r$ is even, then let (without omitting the $00$ at the tail end)
	\begin{align*}
	u & = 10^21^30^4 \ldots 1^{r-1}0^r1^{\lceil (D+r-1)/2 \rceil}0^{\lfloor (D+r+3)/2 \rfloor},\\
	v & = 01^20^31^4 \ldots 0^{r-1} 1^r 0^{D+r+1},
	\end{align*}
	while if $r$ is odd, let
	\begin{align*}
	u & = 10^21^30^4 \ldots 1^{r}0^{D+r+1},\\
	v & = 01^20^31^4 \ldots 0^r 1^{\lceil (D+r-1)/2 \rceil}0^{\lfloor (D+r+3)/2 \rfloor}.
	\end{align*}
	In both cases, we obtain 
	\begin{align*}
	n &= 1+2 + \cdots + r + (r +D -1) = \frac{1}{2} (r^2 + 3r -2) + D,\\
	d(u,v) &= 1+ 2+  \cdots + r + \left\lceil \frac{r + D - 1}{2} \right\rceil  
	= n - \left\lfloor \frac{r+D -1}{2} \right\rfloor.
	\end{align*}
	
	Now, it suffices to prove that $n - \left\lfloor \frac{r+D -1}{2} \right\rfloor > n - \sqrt{2 n}$. 
This holds, if we can prove that $\sqrt{2n} > \frac{r+D -1}{2}$, with $n = \frac{1}{2} (r^2 + 3r -2) + D$ for some $r \geq 1$, and $1 \leq D \leq r+1$. 
But since $D \geq 1$, we have $\sqrt{2n} \geq \sqrt{r^2 + 3r} > r$, and since $D \leq r+1$, we get $\frac{r+D -1}{2} \leq r$. Thus $\sqrt{2n} > \frac{r+D -1}{2}$, and indeed $d(u,v) > n - \sqrt{2n}$.
\end{proof}

We have calculated by computer the exact values of $ \diam (\R_n)$ for $n \leq 30$.
These values, together with the corresponding values of the lower bound of
Theorem~\ref{thm:diam_lower} are shown in Figure~\ref{fig:diameter_plot}.
From the actual values of the diameter we see that in the range given, 
the values for $n = 4, 8, 13, 19 ,26 $ corresponding to $r =2,3,4,5,6$ of 
Lemma~\ref{lem:diam} are actually exact.

\begin{figure}[ht]
	\begin{center}
		\includegraphics[width=0.8\textwidth]{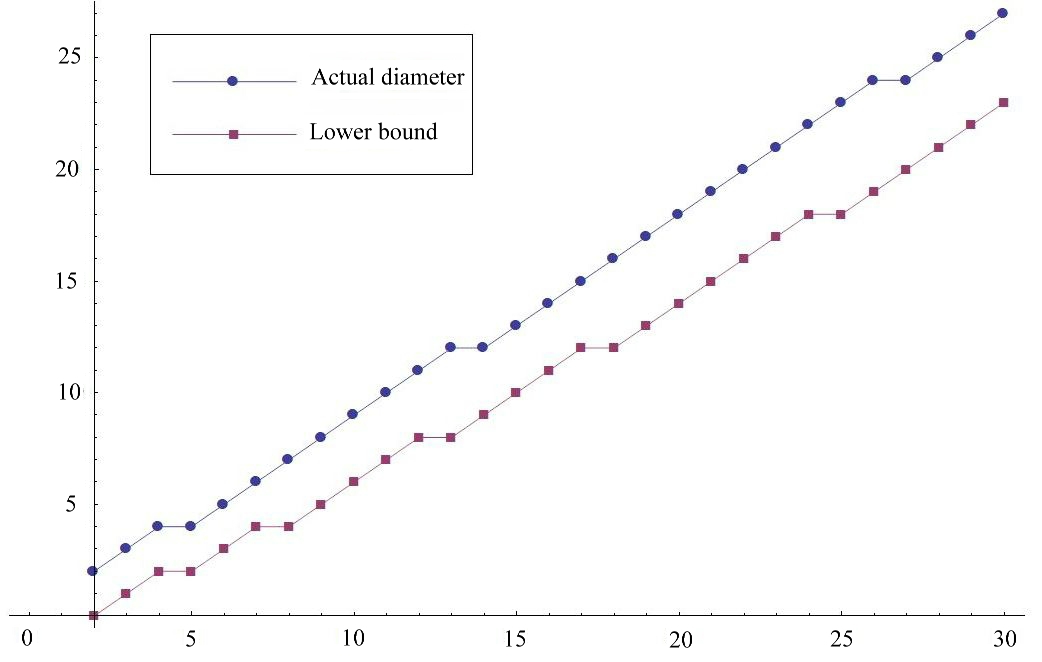}
	\end{center}
	\caption{The exact values and the lower bound given by 
Theorem~\ref{thm:diam_lower} for $\diam (\R_n)$ for $n \leq 30$.}
	\label{fig:diameter_plot}
\end{figure}

Based on this fact and the values in Figure~\ref{fig:diameter_plot},
we conjecture that the diameter is given by
\begin{equation}\label{diameter}
\diam ( \R_n) = n -   \left\lfloor \sqrt{1+ \frac{n}{2}} - \frac{3}{4} \right\rfloor ~.
\end{equation}

\section{Asymptotic results}
\label{sec:asymptotic}

We have seen that $|E(\R_n)| < |E(\Gamma_n)|$, so it is natural to consider 
the asymptotic behavior of the quotient $|E(\R_n)|/|E(\Gamma_n)|$. From the 
recursion in Corollary~\ref{cor:edges-recursion}, expression in equation~\eqref{eq:esges-fib} 
for $|E(\Gamma_n)|$, and the asymptotic behavior of Fibonacci numbers 
it follows that 
\begin{equation}
\label{eq:asymp-edges}
\lim_{n \rightarrow \infty} \frac{|E(\R_n)|}{|E(\Gamma_n)|} =
1-\lim_{n \rightarrow \infty} \frac{|E(\Gamma_{n-4})|}{|E(\Gamma_n)|} 
= \frac{1}{2} (3\sqrt{5}-5) \approx 0.854.
\end{equation}
So asymptotically, $\R_n$ has about four-fifths the number of edges 
of the Fibonacci cube $\Gamma_n$.

The asymptotic average degree of Fibonacci cube $\Gamma_n$ is known~\cite{klavzar+2013} to be 
\begin{equation}
\label{eq:asymp-avgdeg-fib}
\lim_{n \to \infty} \frac{\adeg(\Gamma_n)}{n} =  \left( 1- \frac{1}{\sqrt{5}} \right)  \approx 0.55.
\end{equation}
Using~\eqref{eq:asymp-avgdeg-fib} and~\eqref{eq:asymp-edges}, we obtain 
the asymptotic average degree of the Fibonacci-run graph $\R_n$  as
\begin{equation}\nonumber
\lim_{n \to \infty} \frac{\adeg(\R_n)}{n} = 2 \left( \sqrt{5}  -2 \right) \approx 0.47~,
\end{equation} 
a slightly lower value.
\section{Up--down degree sequences}
\label{sec:degree}

The degree sequences, i.e.\ the nature of the vertices of a given degree in a graph, 
has been well studied for Fibonacci cubes~\cite{klavzar+2011}. Fibonacci-run graphs can also be 
viewed as 
partially ordered sets whose structure is inherited from the 
Boolean algebra of subsets of $[n]$. The elements here correspond to 
all binary strings of length $n$ and the 
covering relation is flipping a 0 to a 1. 
Therefore in $\R_n$ we have a natural distinction between up- and down-degree of a vertex, 
denoted by $ \deg_{up}(v)$ and $\deg_{down}(v)$:
$ \deg_{up}(v)$ is the number of vertices $u$ in $\R_n$ obtained by changing a 0 to a 1, and 
$\deg_{down}(v)$ is the 
number of vertices $u$ in $ \R_n$ obtained from $v$ by changing a 1 to a 0.
We have $ \deg(v) = \deg_{up} (v)  + \deg_{down} (v) $.

The nature of the distribution of the up-degrees and the down-degrees are most easily seen 
from the Hasse diagram of $\R_n$ for which $ \deg_{up}(v)$ and 
$\deg_{down}(v)$ are simply the number of edges emanating up and down from $v \in \R_n$, respectively.

\begin{figure}[ht]
	\begin{center}
		\includegraphics[width=0.45\textwidth]{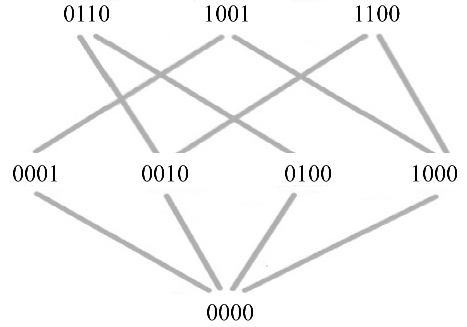}
	\end{center}
	\caption{The Hasse diagram of $\R_4$.}
	\label{fig:Hasse}
\end{figure}

From Figure~\ref{fig:Hasse},
the down-degree enumerator polynomial of $\R_4$ is given by
\begin{equation}\label{down_R4}
1+4d + 3 d^2 ~,
\end{equation}
and its up-degree enumerator polynomial by
\begin{equation}\label{up_R4}
3+2u+2 u^2 + u^4 ~.
\end{equation}
From Figure~\ref{fig:Hasse} we also find that
the degree-enumerator polynomial of $\R_4$ is 
\begin{equation}\label{degree_R4}
5x^2 + 2 x^3 + x^4 ~.
\end{equation}


First, we consider the generating function for the down-degree enumerator polynomials for $\R_n$.
\begin{proposition}
	\label{prop:Rdown}
	The generating function for down-degree enumerator polynomial in Fibonacci-run graphs is
	$$\sum_{n \geq 1} 
	t^n \sum_{v\in \R_n } d^{\deg_{down} (v)} = 
	\frac{t (1 +d +dt +  (d^2-1)t^2  +d(d-1)t^3 + d(d-1)t^4}{1 -t -t^2 -(d-1)t^3 - d(d-1) t^5 }~.$$
\end{proposition}
\begin{proof}
	We first consider the contributions of the strings from the alphabet $R$ of~\eqref{R}. 
The string $100$ contributes 
	$d t^3$, but longer strings $1^k 0^{k+1}$ of length $n$ contribute $d^2t^n$, since both the first and the last $1$ appearing can be switched to $0$. Thus, keeping track of the down-degree by the exponent of $d$ and the total length as the exponent of $t$, the strings from $R$ give 
	$$
	t + d t^3 + d^2 t^5 + d^2 t^7 + \cdots  = t + dt^3 + \frac{d^2 t^5}{1-t^2}~.
	$$
	Therefore the generating function of the free monoid, which includes vertices $V(\R_n), n \geq 1$, is 
	$$
	\frac{1}{
		1 -t - dt^3 - \frac{d^2 t^5}{1-t^2}}~.
	$$
	From this we need to subtract the terms $1, t, t^2$, which correspond to the null word, $0$ and $00$ respectively. 
This gives
	$$
	\frac{1}{
		1 -t - dt^3 - \frac{d^2 t^5}{1-t^2}} - 1 - t - t^2 = 
	\frac{t^3 (1 +d +dt +  (d^2-1)t^2  +d(d-1)t^3 + d(d-1)t^4}{1 -t -t^2 -(d-1)t^3 - d(d-1) t^5 } ~.
	$$
	Finally, we divide by $t^2$ to effectively shorten the length by 2 and get rid of 
the contribution of 
	the last two zeros in each run-constrained string generated to obtain the desired result.
\end{proof}

First few terms given by the generating function in Proposition~\ref{prop:Rdown} are
\begin{align*}
& (1+d) t + (1 + 2d ) t^2 + (1+3d+d^2) t^3 + (1+4d + 3 d^2) t^4 + (1+5d + 7 d^2) t^5\\
&  +(1+6d + 12 d^2 + 2d^3) t^ 6 + (1+7d+19d^2+7d^3) t^7 + \cdots
\end{align*}

Note that since the sum of down-degrees of the vertices in $\R_n$ is the total number 
of edges, we can obtain the generating function of the 
number of edges in $\R_n$ by differentiating the result 
of Proposition~\ref{prop:Rdown} with respect to $d$, and then 
setting $d=1$. This again gives the generating function obtained in Lemma~\ref{lem:edges}. \\

Along similar lines to the 
proof of Proposition~\ref{prop:Rdown}, we also obtain the generating function of the 
up-degree enumerator polynomials for $ \R_n$.
\begin{proposition}
	\label{prop:Rup}
	The generating function for up-degree enumerator polynomials of Fibonacci-run graphs is
$$
\sum_{n \geq 1} 
t^n \sum_{v\in \R_n } u^{\deg_{up}(v)} = \frac{t (1 +u - (u-2) t -2u t^2 + t^3 - (u-1)t^5 - (u-1)t^6 )}
{1- u t - 2 t^2 +(2 u-1) t^3 +t^4 - (u-1) t^5 + (u-1) t^7} ~.
$$
\end{proposition}

First few terms of the power series expansion of the generating function in 
Proposition~\ref{prop:Rup} are
\begin{align*}
& (1+u) t + (2+u^2 ) t^2 + (2+2u+u^3) t^3 + (3+2u+2u^2+u^4) t^4 \\ 
& + (5+2u+3u^2+2u^3+u^5) t^5 +(6+6u+2u^2+4u^3+2u^4+u^6) t^6 + \cdots
\end{align*}

It is desirable to study the generating function of the bivariate up-down degree enumerator 
polynomials
$$\sum_{v\in \R_n } u^{\deg_{up}(v)} d^{\deg_{down}(v)}~.$$
For example, for $n=4$, this polynomial is 
\begin{equation}\label{up-down_R4}
3 d^2 + 2du +2 du^2+u^4 ~,
\end{equation}
as can be verified by inspecting $\R_4$ in Figure
~\ref{fig:Hasse}.  The polynomial in~\eqref{up-down_R4} specializes to~\eqref{down_R4} for $u=1$, to~\eqref{up_R4} for $ d=1$, and to~\eqref{degree_R4} for $ u=d=x$.
More generally, the generating function of the up-down degree enumerator polynomials for 
$\R_n$ would give 
Proposition~\ref{prop:Rdown} and 
Proposition~\ref{prop:Rup} 
as corollaries, and provide the generating function for the degree enumerator polynomials 
itself for $ u=d=x$. The derivation of this general case is of 
independent interest and is studied in detail in a companion paper~\cite{paper2}. 

\section{Relation to hypercubes}
\label{sec:hypercubes}

Recall that partial cubes are an important and well-studied class of graphs. It is clear that Fibonacci-run graphs are subgraphs of hypercubes. 
So it is natural to ask whether they can be isometrically embedded in a hypercube. 

\begin{proposition}
	Fibonacci-run graph $\R_n$ is a partial cube if and only if $n \leq 6$. 
\end{proposition}

\begin{proof}
	We can check with a computer that graphs $\R_n$ for $n \leq 6$ are partial cubes.
	
	Let $n \geq 7$. Winkler~\cite{winkler1984partial} proved that a connected graph is a partial cube if and only if it is bipartite and the Djokovi\'{c}-Winkler relation $\theta$ is transitive. Recall that for edges $xy, uv$ of a bipartite graph it holds $xy \, \theta \, uv \iff d(x,u) = d(y,v) \text{ and } d(x,v) = d(y,u)$. Let $a,b$ be arbitrary  run-constraint strings of appropriate lengths and take vertices in $\R_n$:
	\begin{align*}
	x = & a 111100000 b, \quad y = a 111000000b, \\
	p = & a 011100000 b, \quad q = a 011000000 b, \\
	u = & a 100100000 b, \quad v = a 100000000 b. 
	\end{align*}
	It can be checked that $xy, pq, uv \in E(\R_n)$, and that $xy \, \theta \, pq$ and $pq \, \theta \, uv$, but not $xy \, \theta \, uv$. Thus $\theta$ is not transitive and $\R_n$ is not a partial cube.
\end{proof}

Additionally, we remark that the Fibonacci-run graphs are in general not median graphs. They are median for $n \leq 4$, as those graphs are isomorphic to $\Gamma_n$, which are known to be median~\cite{klavzar2005}. However, $\R_n$ is not a median graph for $n \geq 5$. Consider the following triple of 
vertices:
\begin{align*}
u = & 111000^{n-5}, \\
v = & 100000^{n-5}, \\
w = & 001000^{n-5}~. 
\end{align*}
Clearly, $u,v,w\in V(\R_n)$. Since $\R_n$ is an isometric subgraph of a hypercube, the unique candidate for a median of $u, v, w$ is obtained in the same way as in a hypercube -- using the 
majority rule~\cite{imrich2000}. With this rule, a coordinate of the median vertex 
equals the value that appears at least in two of the vertices $u,v,w$ on that coordinate. 
Thus the unique candidate for a median is $ x = 101000^{n-5} \notin V(\R_n)$, hence $\R_n$ is not a median graph.

Next, we consider an analogous question. 
Instead of observing Fibonacci-run graphs as subgraphs of hypercubes, we consider 
the possible hypercubes which are subgraphs of a Fibonacci-run graph. A similar 
question has been studied for Fibonacci cubes as well, cf.\ results about cube polynomial 
listed in~\cite{klavzar2013-survey}, and generalizations in~\cite{Saygi2017} .

\begin{proposition}
	\label{prop:gCubes}
	If $h_{n,d,k}$ denotes the number of $k$-dimensional hypercubes $Q_k$ in 
	$\R_n$ whose distance to the all zero vertex $0^n$ is $d$, then the generating function 
$\displaystyle \sum_{n\geq 1} t^n \sum_{ d, k \geq 0 }
	h_{n,d,k} q^d x^k $ equals
	{\small
		$$
		\frac{t (1 +q+x +(q+x)t +  ((q+x)^2-1)t^2  +(q+x)(q+x-1)t^3 + (q+x)(q+x-1)t^4}
		{1 -t -t^2 -(q+x-1)t^3 - (q+x)(q+x-1) t^5 } ~.
		$$
	}
\end{proposition}
\begin{proof}
	Let us keep track of all hypercubes $Q_k$ in $\R_n$. To each such 
subgraph, we associate the monomial $q^d x^k$, where $d$ is the distance of the $Q_k$ 
to the all zero vertex $0^n \in V(\R_n)$. For every vertex $v\in \R_n$, 
select $k$ $1$s in its string representation that can be replaced with a $0$. Note 
that the number of such $1$s is at most 
$\deg_{down} (v)$ (which is not necessarily equal to $\deg(v)$). 
By flipping these $1$s to $0$s in all possible ways, we obtain the vertices of 
a copy of $Q_k$ in $\R_n$. The distance of this hypercube 
	to $0^n$ is $w-k$, where $w= |v|_1$ is 
the Hamming weight of the vertex $v$. So from every vertex of Hamming 
weight $w$ and down-degree $r$, we obtain ${r\choose k}$ different copies of 
$Q_k$, with the associated monomial $q^{w-k} x^k$ for each. 
	
	The generating function can be obtained from Proposition~\ref{prop:Rdown} 
by replacing each $d^r$ that appears in the series expansion, by $(q+x)^r$. 
The reason for this is that the term $d^r$ arises from a vertex $v$ with down-degree $r$, 
and its contribution to the monomials making up the generating function is 
	$$
	\sum_{k=0}^r {r \choose k } q^{r-k} x^k = 
	(q+x)^r~.
	$$
	Thus this replacement is the same as if we replace $d$ by $q+x$ in the 
generating function in Proposition~\ref{prop:Rdown}, and this substitution gives the generating function of
the proposition. 
\end{proof}

The coefficients of the generating function in Proposition~\ref{prop:gCubes} for $ 1\leq n \leq 6 $ are
\begin{align*}
	&1+q+x,\\
	&1+2q + 2x,\\
	&1+3q+q^2 + (3 +2q) x + x^2,\\
	&1+4q+3q^2+ (4+6q)x + 3 x^2,\\
	&1+5q+7q^2+(5+14q)x + 7 x^2, \\
	& 1 + 6q + 12q^2 + 2q^3 + (6+24q+6q^2)x +(12+6q)x^2 + 2 x^3.
\end{align*}

So for example, the term $(12+6q)x^2$ in the last polynomial indicates that in $\R_6$, there are 
12 squares ($Q_2$'s) that contain the all zero vertex, and 6 squares whose distance 
to the all zero vertex is one.

\section{Hamiltonicity and further directions}
\label{sec:hamilton}

It is well known and follows from the fundamental decomposition of Fibonacci cubes, that $\Gamma_n$ has a Hamiltonian path for every $n \geq 0$~\cite{cong+1993}. However, surprisingly, we were not able to obtain a similar result for Fibonacci-run graphs. Properties of the graphs $\R_n$ for $1 \leq n \leq 12$ are presented in Table~\ref{tbl:hamilton}.

\begin{table}[ht]
	\begin{center}
		\begin{tabular}{c||c|c|c|c|c|c|c|c|c|c|c|c}
			$n$ & 1 & 2 & 3 & 4 & 5 & 6 & 7 & 8 & 9 & 10 & 11 & 12 \\ \hline
			cycle? & no & no & no & \checkmark & no & no & \checkmark & no & no & \checkmark & no & no \\ \hline
			path? & \checkmark & \checkmark & \checkmark & \checkmark & \checkmark & \checkmark & \checkmark & \checkmark & \checkmark & \checkmark & \checkmark & \checkmark \\
		\end{tabular}
	\end{center}
	\caption{Graphs $\R_n$ for $1 \leq n \leq 12$ which have a Hamiltonian cycle or a Hamiltonian path.}
	\label{tbl:hamilton}
\end{table}

Since the graphs $\R_n$ are bipartite (vertices with odd/even Hamming weight form the two parts), the graph can only be Hamiltonian if both parts are of equal size. Let $\Delta_n$ denote the number of vertices of even Hamming 
weight minus the 
number of vertices of odd Hamming weight in $\R_n$. Using Corollary~\ref{cor:vertices-weight}, we calculate that
$$
\Delta_n = 
\begin{cases}
0 & \text{if } n \equiv 1 \pmod{3},\\
-1  & \text{if } n \equiv 2,3 \pmod{6},\\
1 & \text{if } n \equiv 0, 5 \pmod{6}.
\end{cases}
$$

Therefore we have
\begin{lemma}
	\label{lem:hamilton_cycle}
	If $n \not\equiv 1 \pmod{3}$, then $\R_n$ does not contain a Hamiltonian cycle.
\end{lemma}

Even though this lemma does not solve the Hamiltonicity of Fibonacci-run graphs completely, 
based on the decomposition in Figure~\ref{fig:decomposition},
and the the computational values we have obtained in 
Table~\ref{tbl:hamilton}, we  conjecture that 
$\R_n$ is  Hamiltonian if and only if $n \equiv 1 \pmod{3}$ and that
$\R_n$ always has a Hamiltonian path. 
We remark that 
the decomposition of $ \R_n$ in terms of smaller dimensional Fibonacci-run graphs 
as indicated in Figure~\ref{fig:decomposition}
indicates various approaches for the construction of the Hamiltonian paths 
inductively.
The decomposition would require the 
entry and the exit vertices of the Hamiltonian path for each smaller dimensional Fibonacci-run 
graph in a uniform manner. Such a construction is
a topic of further interest. \\


Additionally,
it would be satisfying to find a proof of the expression
in Lemma~\ref{lem:edges} for the number of edges combinatorially, without 
using the decomposition of $\R_n$, and also prove the conjecture on the exact value 
of the diameter given by the expression in~\eqref{diameter}.
There is also the problem of the determination of the radius of $\R_n$.\\

We also conjecture that the lower bound in Lemma~\ref{lem:independence}
is actually tight.
Of course 
our conjecture on the existence of Hamiltonian paths, if true,  would automatically solve this problem.

\section*{Acknowledgments}
We would like to thank Prof. Sandi Klav\v{z}ar 
for enabling the cooperation of the authors of this article, and his early interest in the topic.
The first author would like to acknowledge the hospitality of Reykjavik University during his 
sabbatical stay there, during which a portion of this research was carried out.

\newpage

\section*{Appendix: Figures of some Fibonacci-run graphs}
\begin{figure}[ht]
	\centering
	\includegraphics[width=0.4\textwidth]{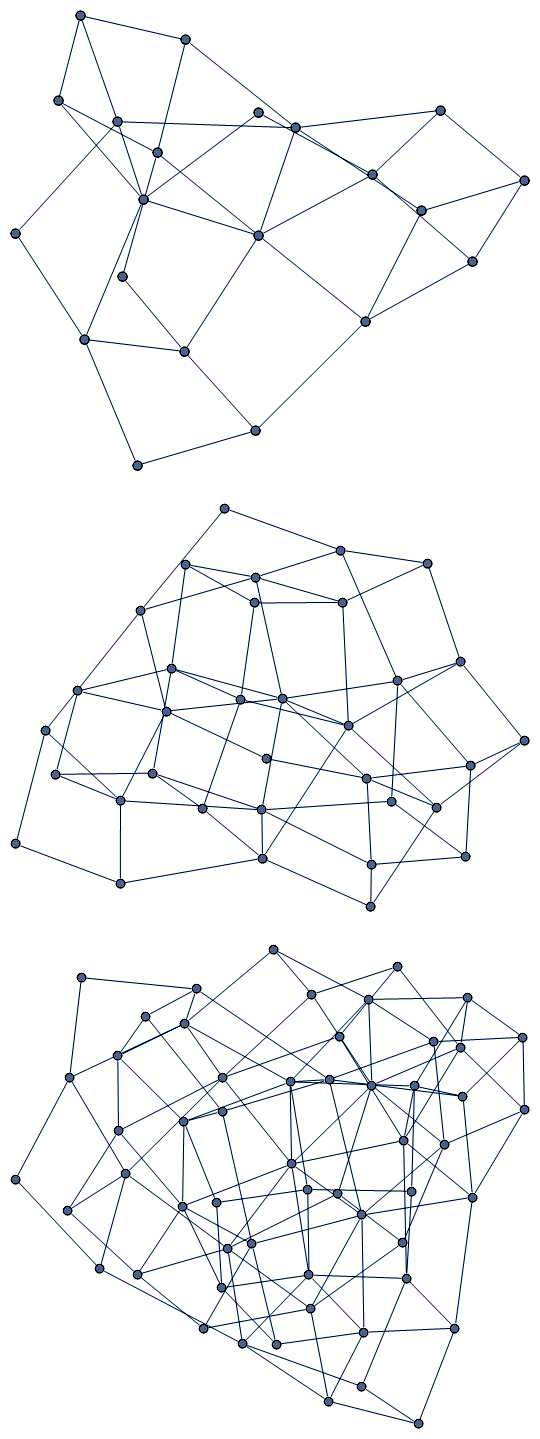}
	\caption{Fibonacci-run graphs $\R_6$, $\R_7$ and $\R_8$.\label{R8-8}}
\end{figure}

\begin{figure}[ht]
	\centering
	\includegraphics[width=0.5\textwidth]{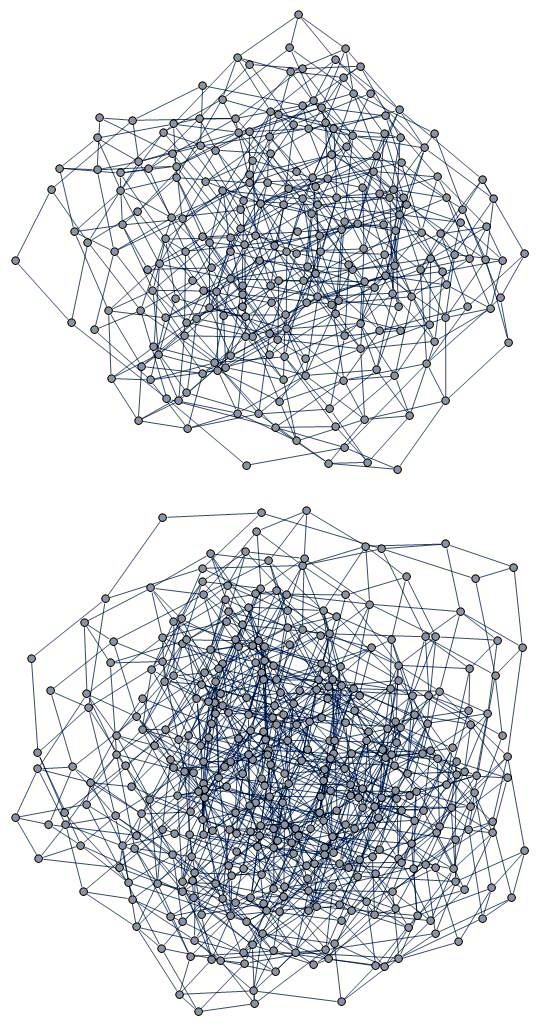}
	\caption{Fibonacci-run graphs $\R_{11}$ and $ \R_{12}$.\label{R11-12}}
\end{figure}

\end{document}